\documentclass[11pt]{article}
\input epsf
\usepackage{amsfonts}
\usepackage{amscd}
\usepackage{amsmath,amssymb,amsthm}
\usepackage{amstext}
\usepackage{amscd}
\usepackage{graphicx}
\usepackage[all]{xy}
\usepackage{pstricks,pst-node,pst-tree}

\theoremstyle{plain}
\newtheorem{lem}{Lemma}[section]
\newtheorem*{case}{Theorem 0}
\newtheorem*{main}{Theorem 1}
\newtheorem*{LA}{Lemma A}
\newtheorem*{LB}{Lemma B}

\newtheorem{theo}[lem]{Theorem}

\newtheorem{coro}[lem]{Corollary}

\theoremstyle{definition}

\newtheorem{definition}[lem]{Definition}

\newtheorem{rem}[lem]{Remark}

\renewcommand{\descriptionlabel}[1]%
       {\hspace{\labelsep}\textsf{#1}}

\newcommand{\R}{\mathbb{R}}

\newcommand{\noi} {\noindent}

\begin{document}
\title{Cubulated moves for 2-knots.}
\author{ Juan Pablo D\'iaz\thanks{This work was partially supported by CONACyT (M\'exico), FORDECYT 265667},
 Gabriela Hinojosa\thanks{This work was partially supported by CONACyT (M\'exico), CB-2009-129939.}, Alberto Verjosvky, \thanks{This work was partially supported by
CONACyT (M\'exico), CB-2009-129280 and  PAPIIT (Universidad
Nacional Aut\'onoma de M\'exico) \#IN 106817.}}
\date{August 21, 2017}

\maketitle

\begin{abstract} \noi In this paper, we prove that given two cubical links of dimension two in $\mathbb{R}^4$,
they are isotopic if and only if one can pass from one to the other
by a finite sequence of cubulated moves. These moves are analogous
to the Reidemeister and Roseman moves for classical tame knots of dimension one and two, respectively. 
\end{abstract}


\section{Introduction}

In \cite{BHV} it was shown that any smooth knot ${K}^n:{\mathbb S}^n\hookrightarrow{\mathbb R}^{n+2}$ can be deformed isotopically into the $n$-skeleton  
of the canonical cubulation of  ${\mathbb R}^{n+2}$ and this isotopic copy is called \emph{cubical $n$-knot}. In particular, every smooth 1-knot 
$\mathbb S^1\subset{\mathbb R}^3$ is isotopic to a cubical knot.   

\noindent There are two types of elementary ``cubulated moves''. The first one (M1) is obtained by dividing each cube of the original cubulation of $\mathbb{R}^3$ into $m^3$ cubes, 
which means that each edge of the knot is subdivided into $m$ equal segments. The second one (M2) consists in exchanging a connected set of edges in a face of the 
cubulation 
with the complementary edges in that face. If two cubical knots $K_1$ and $K_2$ are such that we can convert $K_1$ into $K_2$ using a finite sequence of cubulated moves then we say that they are \emph{equivalent via cubulated moves} and is denoted by  $K_1\overset{c}\sim K_2$.   

\noindent In \cite{HVV} it was  proved the following:

\begin{case}\label{case}
Given two cubical knots $K_1$ and $K_2$  in $\mathbb{R}^{3}$, $K_1\cong K_2 \cong \mathbb{S}^{1}$, 
they are isotopic if and only if $K_1$ is equivalent to $K_2$ by a finite sequence of cubulated moves; {\emph{i.e.}}, $K_1\sim  K_2 \iff K_1\overset{c}\sim K_2$.
\end{case}

\noindent Theorem 0 is analogous to the Reidemeister moves of classical tame knots for cubical knots.

\begin{figure}[h] 
 \begin{center}
 \includegraphics[height=4cm]{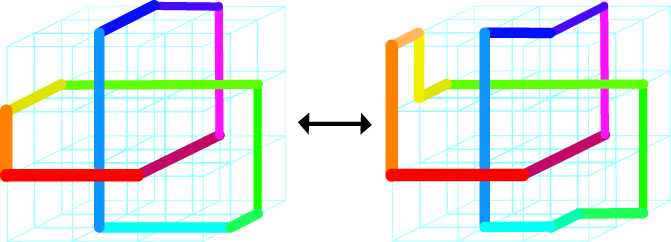}
\end{center}
\caption{\sl Two isotopic knots  are equivalent via cubulated moves.} 
\label{M1}
\end{figure} 

\noindent Notice that cubulated moves can be extended to cubical 2-knots  in a natural way: 
The first one (M1) is obtained by dividing each hypercube of the original cubulation of 
$\mathbb{R}^4$ into $m^4$ hypercubes, and the second one (M2) consists in exchanging a connected set of squared faces homeomorphic to a 
disk $\mathbb{D}^2$ in a cube of the cubulation (or a subdivision of the cubulation) with the complementary faces in that cube.  

\noindent The study of 2-knots in $\R^4$ has been considered by various authors, for instance in 
\cite{CRS}, \cite{CS}, \cite{K} and \cite{Roseman}.

\noindent Our goal is to extend the Theorem 0 for cubical knots of dimension two:

\begin{main}\label{main}
Given two cubical 2-knots $K^2_1$ and $K^2_2$  in $\mathbb{R}^{4}$,
 then they are isotopic if and only if $K^2_1$ is equivalent to $K^2_2$ by cubulated moves; {\emph{i.e.}}, $$K^2_1\sim  K^2_2 \iff K^2_1\overset{c}\sim  K^2_2.$$
\end{main}

\section{Preliminaries}

\subsection{Cubulations of $\mathbb{R}^{4}$}
\noi The regular hypercubic honeycomb whose Schl\"afli symbol is $\{4,3,3,4\}$, is called a {\it cubulation} of $\mathbb{R}^{4}$. In other words, a cubulation of  $\mathbb{R}^{4}$ is a decomposition into a collection of right-angled $4$-dimensional
hypercubes $\{4,3,3\}$ called  $\textit{cells}$ such that any two are
either disjoint or meet in one common $k-$face of some dimension $k$. This provides  $\mathbb{R}^{4}$ with the structure of a cubical
complex whose category is similar to the simplicial category PL.   

\noi The combinatorial structure of the regular Euclidean honeycomb $\{4,3,3,4\}$ is the following: around each vertex  there are 8 edges, 24 squares, 32 cubes and 16 hypercubes. Around each edge there are 6 squares, 12 cubes and 8 hypercubes. Around each square there are 4 cubes and 4 hypercubes. Finally around each cube there are  2 hypercubes.

\begin{figure}
\centering
\includegraphics[scale=0.3]{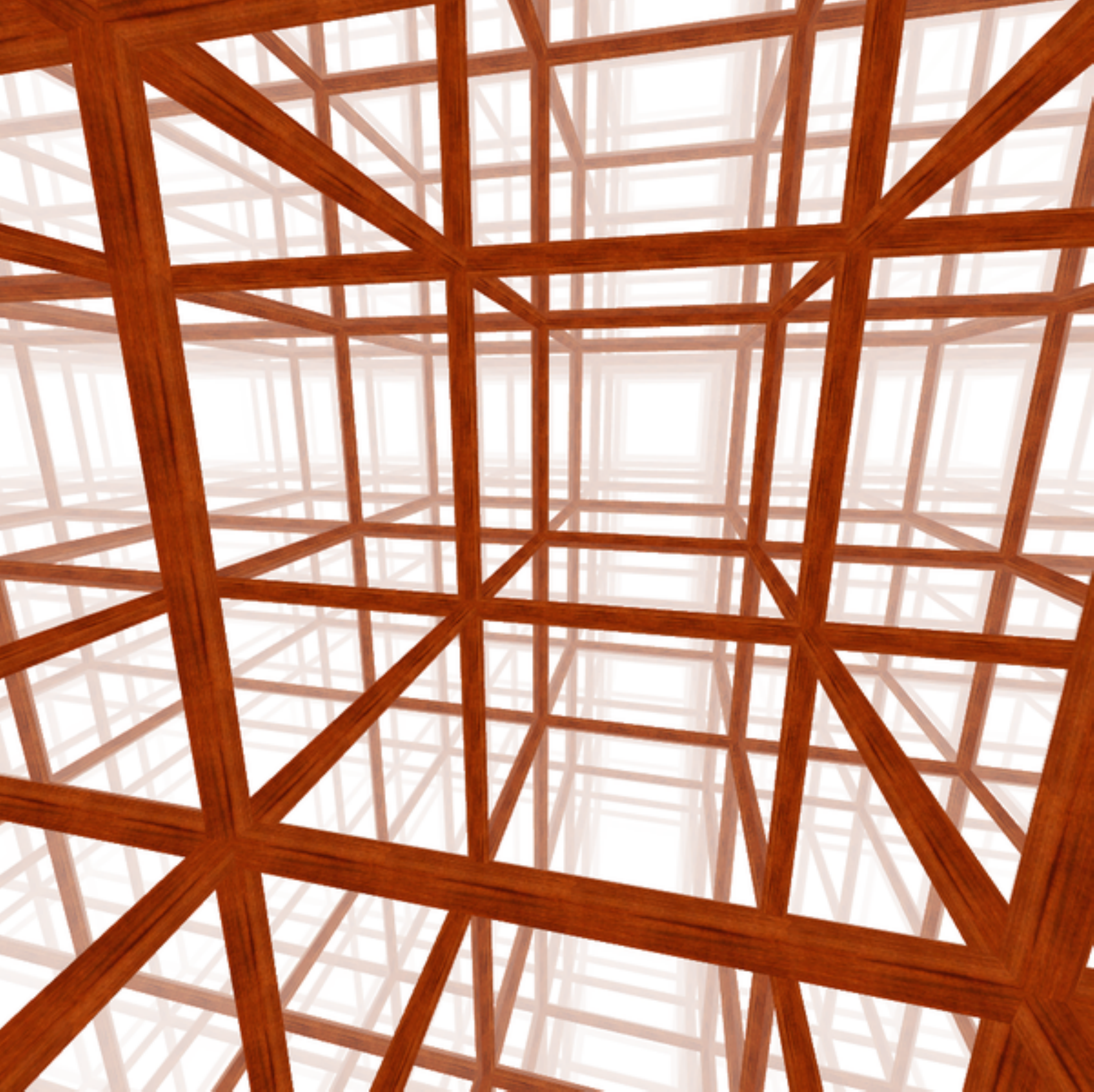}
\begin{center}
{{\bf Figure 1.} The 3-dimensional cubical kaleidoscopic honeycomb $\{4,3,4\}$. 
This figure is courtesy of Roice Nelson \cite{RN}.}
\end{center}
\end{figure}

\noindent The canonical hypercubic honeycomb $\cal C$ of $\mathbb{R}^{4}$ is its decomposition into hypercubes which are the images of the unit hypercube:
$$\{4,3,3\}=I^4=[0,1]^4=\{(x_{1},x_{2},x_{3},x_{4})\in \mathbb{R}^4 \,|\,0\leq x_{i}\leq 1\}$$ by translations by vectors with integer coefficients. Then all vertices of $\cal C$ have integer coordinates.   

\noindent Any regular hypercubic honeycomb $\{4,3,3,4\}$ or cubulation of $\mathbb{R}^{4}$ is obtained from the canonical cubulation by applying a conformal transformation  to the canonical cubulation. Remember that a conformal transformation is of the form  
$x\mapsto{\lambda{A(x)}+a},$ where $\lambda\neq0,\,\,a\in \mathbb{R}^{4},\,\,\,A\in{SO(4)}$. 

\begin{definition} The $k${\it-skeleton} of $\cal C$, denoted by $\mathcal{S}^k$, consists of the union of the $k$-skeletons of the hypercubes in  $\cal C$,
{\it i.e.,} the union of all cubes of dimension $k$ contained in the faces of the $4$-cubes in  $\cal C$.
We will call the 2-skeleton $\mathcal{S}^2$ of $\cal C$ the {\it canonical scaffolding} of $\mathbb{R}^{4}$.
\end{definition}

\noindent Notice that all the previous definitions can be extended  in a natural way to $\mathbb{R}^{n+2}$.

\subsection{Cubical and smooth 2-knots}

\noindent In classical knot theory, a subset $K$ of a space $X$ is a {\it knot} if $K$ is homeomorphic to a  $p$-dimensional sphere 
$\mathbb{S}^{p}$ embedded in either the Euclidean $n$-space $\mathbb{R}^{n}$ or the $n$-sphere $\mathbb{S}^{n}=\mathbb{R}^{n}\cup\{\infty\}$, where $p<n$. Two knots $K_1$, $K_2$ are {\it equivalent} or {\it isotopic} if there is a homeomorphism $h:X\hookrightarrow X$ such that $h(K_1)=K_2$;
in other words $(X,K_1)\cong (X,K_2)$. However, a knot $K$ is sometimes defined to be an embedding
$K:\mathbb{S}^{p}\hookrightarrow\mathbb{S}^{n}\cong\mathbb{R}^{n} \cup \{\infty \}$ (see \cite{mazur}, \cite{rolfsen}).
We shall also find this convenient at times and will use the same symbol to
denote either the map $K$ or its image $K(\mathbb{S}^{p})$ in $\mathbb{S}^{n}$.

\begin{definition}
Let $K^2$ be a $2$-dimensional knot in $\mathbb{R}^{4}$. If $K^2$ is contained in  $\mathcal{S}^2$, we say that $K^{2}$ is a \emph{cubical knot}. 
\end{definition}

\begin{definition} If $K^2$ is a smooth knot and $\widehat{K^2}$ denotes a cubical knot which is isotopic to $K^2$, {\it i.e.,} $K^2  \sim \widehat{K^2}$, we say that $\widehat{K^2}$ is a \emph{cubical diagram} of the knot $K^2$.
\end{definition}

\noi Given two parametrized $2$-dimensional smooth knots $K^2_1$, $K^2_2:\mathbb{S}^2\hookrightarrow\mathbb{R}^{4}$ we say they are \textit{smoothly isotopic} if there exists a 
smooth isotopy $H:\mathbb{S}^{2}\times\mathbb{R}\rightarrow\mathbb{R}^{4}$ such that

$$
H(x,t)= \left\{ \begin{array}{ll}
K^2_1(x) & \mbox{if $t\leq1$},\\
K^2_2(x) & \mbox{if $t\geq 2$},
\end{array} \right.
$$

and $H(\ \cdot \ ,t)$ is an embedding of $\mathbb{S}^2$ for all $t\in\mathbb{R}$.  

\begin{definition}\label{cylinder}
We will say 
$J^3=\{(H(x,t),t)\in\mathbb{R}^{5}\,|\,x\in\mathbb{S}^{2},\,\,t\in\mathbb{R}\}$ is the {\emph {isotopic cylinder}} of $K^2_1$ and $K^2_2$. 
\end{definition}

\noi Note that $J^3$ is a smooth noncompact submanifold of codimension two in $\mathbb{R}^{5}$.  

\begin{definition}
Let $p:\mathbb{R}^{5}\hookrightarrow\mathbb{R}$ be the projection onto the last coordinate. Let $M$ be a connected subset of $\mathbb{R}^{5}$ such that $p^{-1}(c)\cap M$ (or $p|_M^{-1}(c)$) is connected for all $c\in\mathbb{R}$,
we say that $M$  is {\emph{sliced by connected level sets of $p$}}. 
\end{definition}

\noindent Observe that there is no restriction on the dimension of $M$.

\noindent In  \cite{HVV} it was proved the following result.

\begin{theo}\label{trace}
The isotopic cylinder $J^3$ can be cubulated. In other words, there exists an isotopic copy $\widehat{J^3}$ of $J^3$ contained
in the $3$-skeleton of the canonical cubulation of $\mathbb{R}^{5}$. 
Moreover $\widehat{J^3}$ can be chosen to be sliced by connected level sets of $p$.
\end{theo}
 
\noi Consider $\widehat{J^3}$ as above; so $\widehat{J^3}$ is a cubical $3$-manifold and  there exist integer numbers 
$m_1$ and $m_2$ and cubical knots $\widehat{K^2_1}$ and $\widehat{K^2_2}$ isotopic to $K^2_{1}$ and $K^2_{2}$, respectively;
such that $p|_{\widehat{J^3}}^{-1}(t)= \widehat{K^2_1}$ for all $t\leq m_1$ and
$p|_{\widehat{J^3}}^{-1}(t)= \widehat{K^2_2}$ for all $t\geq m_2$.  

\subsection{Cubulated moves}

\begin{definition} The following are the allowed {\emph{cubulated moves}}:
\begin{description}
\item[{\bf{M1}}] {\emph{Subdivision:}} Given an integer $m>1$, consider the subcubulation  ${\cal{C}}_m$ of $\cal{C}$ by subdividing each $k$-dimensional cell of $\cal{C}$ in 
$m^k$ congruent $k$-cells in ${\cal{C}}_m$, in particular each hypercube in $\cal{C}$ is subdivided in $m^4$ congruent hypercubes in ${\cal{C}}_m$. Moreover as a cubical complex, each $k$-dimensional face of the 2-knot $K^2$ is subdivided into $m^k$ congruent 
$k$-faces. Since ${\cal{C}}\subset {\cal{C}}_{m}$, then
$K^2$ is contained in the scaffolding ${\cal{S}}^2_m$ (the $2$-skeleton) of ${\cal{C}}_m$. 
 
\item[{\bf{M2}}] {\emph{Face Boundary Moves:}} Suppose that $K^2$ is contained in some subcubulation ${\cal{C}}_m$ of  
the canonical cubulation ${\cal{C}}$ of $\mathbb{R}^{4}$. 
Let $Q^4\in {\cal{C}}_m$ be a $4$-cube such that
$A^2=K^2\cap Q^4$ contains a $2$-face. We can assume, up to applying the elementary $(M1)$-move if necessary, that $A^2$ consists of either one, 
two, or three
squares such that it is a connected surface and it is contained in the boundary of a 3-cube
$F^3\subset Q^4$; in other words $A^2$ is a cubical disk contained in the boundary of
$F^3$. The boundary $\partial F^3$ is divided by $\partial A^2$ into two cubulated surfaces, one of them is $A^2$ and we denote the other by $B^2$. Observe
that both  cubulated surfaces share a common circle boundary. The face boundary move  consists in replacing $A^2$ by $B^2$ (see Figure \ref{M1}). There are three types of face boundary moves depending of the number of 2-faces in each $A^2$ and $B^2$. If $A^2$ has $p$ squares then $B^2$ has $6-p$ squares. 

\begin{figure}[h] 
 \begin{center}
 \includegraphics[height=5cm]{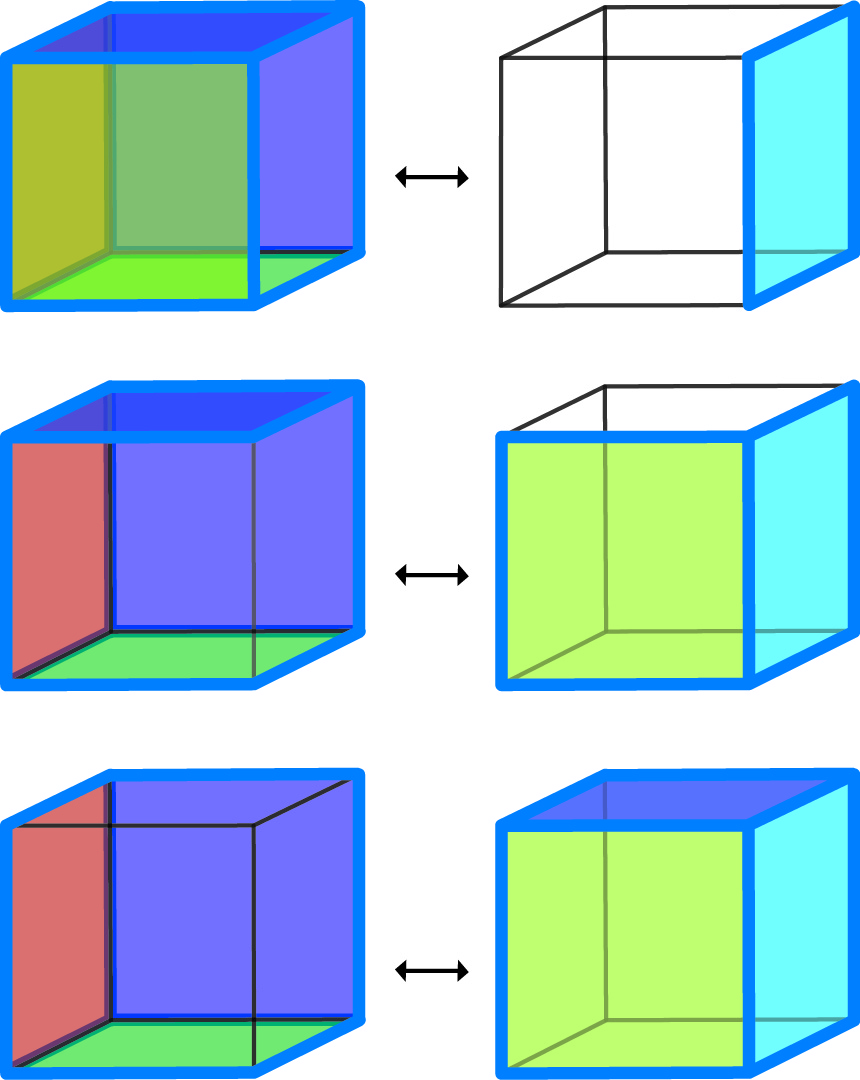}
\end{center}
\caption{\sl The three types of face boundary moves.} 
\label{M1}
\end{figure} 
\end{description}
\end{definition}

\begin{rem}\label{equivalencia}
It can be easily shown that the $(M2)$-move can be extended to an ambient isotopy of $\mathbb{R}^4$ .
\end{rem}

\begin{definition}
Given two cubical $2$-knots $K^2_1$ and $K^2_2$ in $\mathbb{R}^{4}$. We say that
$K^2_1$ is {\emph{equivalent}} to $K^2_2$ {\emph{by cubulated moves}}, denoted by 
$K^2_1\overset{c}\sim K^2_2$, if we can transform $K^2_1$ into $K^2_2$ by a finite number
of cubulated moves. 
\end{definition}

\section{Main theorem}

\noi We are now ready to prove our main theorem. Notice that this proof is a generalization of the one given in \cite{HVV} for the one dimensional case.

\begin{main}\label{main}
Given two cubical 2-knots $K^2_1$ and $K^2_2$  in $\mathbb{R}^{4}$,
 then they are isotopic if and only if $K^2_1$ is equivalent to $K^2_2$ by cubulated moves; {\emph{i.e.}}, $$K^2_1\sim  K^2_2 \iff K^2_1\overset{c}\sim  K^2_2.$$
\end{main}

\begin{proof} \noi First, note that if $K^2_1$ and $K^2_2$ are equivalent by cubulated moves, then
these 2-knots are isotopic by Remark \ref{equivalencia}.  Hence, what remains to be proved is
that two cubical 2-knots that are isotopic must also be equivalent by cubulated
moves.  

\noi Our strategy is like the one used for cubic knots of dimension one (see \cite{HVV}).  First, for $i \in \{1,2\}$, we will smooth each $K^2_i$ to obtain $\widetilde{K^2_i}$, and then cubulate these two 2-knots to obtain $\widehat{K^2_i}$ in such a way that 

A) $K^2_i \overset{c}\sim  \widehat{K^2_i} \quad$ and 

B) $\widehat{K^2_1}\overset{c}\sim  \widehat{K^2_2}$.   

\noindent In \cite{DHVV} was proved the following.    

\noindent{\bf Theorem.} \emph{Any compact, closed, oriented, cubical surface $M^2$ in $\mathbb{R}^{4}$ is smoothable. More precisely, $M^2$
admits a global continuous transverse field of 2-planes and therefore by a theorem of J. H. C. Whitehead there is an arbitrarily small topological 
 isotopy that moves $M^2$ onto a smooth surface $\widetilde{M^2}$ in $\R^4$ (see  \cite{pugh}, \cite{whitehead})}.

\noindent As a consequence, we have that  given a cubical knot $K^2$, there exists a smooth knot $\widetilde{K^2}$ isotopic to
$K^2$ such that $\widetilde{K^2}$ is ${\cal{C}}^0$-arbitrarily close to $K^2$.  

\noi Let $J^3$ be the isotopic cylinder (see Definition \ref{cylinder}) of $\widetilde{K^2_1}$ and $\widetilde{K^2_2}$.
Then $J^3$ is a smooth submanifold of codimension two in $\mathbb{R}^{5}$. 
By Theorem \ref{trace}, there exists an isotopic copy of $J^3$, say $\widehat{J^3}$, contained in the $3$-skeleton of the canonical 
cubulation $\cal{C}$ of $\mathbb{R}^{5}$. Recall that $\widehat{J^3}$ is sliced by connected level sets of $p$ and  there exist integer numbers 
$m_1$ and $m_2$ such that $p^{-1}(t)\cap \widehat{J^3} =\widehat{K^2_1}$ for all $t\leq m_1$ and
$p^{-1}(t)\cap \widehat{J^3} = \widehat{K^2_2}$ for all $t\geq m_2$, where $\widehat{K^2_1}$ and $\widehat{K^2_2}$ are cubical knots which are isotopic to $\widetilde{K^2_{1}}$ and $\widetilde{K^2_{2}}$, respectively.  

\noindent Now, we will use the following two results which will be proved in Sections \ref{smooth} and \ref{cubic}, respectively.  

\begin{LA}\label{smooth1} 
Given a cubical 2-knot $K^2$, we can choose a small cubulation ${\cal{C}}_{m}$  fine enough so that 
$N^4_{K}=\{Q^4\in {\cal{C}}_{m}\,|\,Q^4\cap K^2\neq\emptyset\}$ is a closed tubular neighborhood of $K^2$ and
$Q^4\cap K^2$ is equal to either a vertex, one square, two squares sharing an edge (neighboring 2-faces) or three neighboring squares (two by two neighboring 2-faces). 
We can also choose $\widetilde{K^2}$ isotopic to $K^2$ such that $\widetilde{K^2}$ is 
${\cal{C}}^0$-arbitrarily close to $K^2$ and $\widetilde{K^2}\subset \mbox{Int}(N^4_{K})$. Let 
$\widehat{K^2}$ be an isotopic copy of $\widetilde{K^2}$ contained in $\partial N^4_{K}$, then $K^2\overset{c}\sim  \widehat{K^2}$; {\emph{i.e.}}, we can
go from $K^2$ to $\widehat{K^2}$ by a finite sequence of cubulated moves.
\end{LA}

\begin{rem}\label{conditions}
\begin{enumerate}
\item The existence of $\widehat{K^2}$ is proved on Theorem 3.1 in \cite{BHV} (see also the proof of Theorem \ref{trace}).
\item We may assume, using a subdivision move if necessary, that the intersection $Q^4\cap K^2$ is equal to either a vertex, one square, two squares sharing an edge (neighboring 2-faces) or three neighboring squares (two by two neighboring 2-faces). 
\end{enumerate}
\end{rem}

\begin{LB}\label{cubemoves} 
Given two cubical knots $K^{2}_1$ and $K^{2}_2$, we obtain $\widehat{K^{2}_{1}}$ and $\widehat{K^{2}_{2}}$ as in Lemma A. Then
there exists a finite sequence
of cubulated moves that carries $\widehat{K^{2}_{1}}$ into $\widehat{K^{2}_{2}}$. In other words, $\widehat{K^{2}_{1}}$ is equivalent to $\widehat{K^{2}_{2}}$ by cubulated moves: $\widehat{K^{2}_{1}}\overset{c}\sim  \widehat{K^{2}_{2}}$
\end{LB}

\noindent If we go back to the proof of the Theorem 1, we have by Lemma A that there exist a finite sequence of cubulated moves that carries $K^{2}_{1}$ into $\widehat{K^{2}_{1}}$ and also a finite
sequence of cubulated moves that transforms $K^{2}_{2}$ into $\widehat{K^{2}_{2}}$. By Lemma B,
 there exists a finite sequence of cubulated moves that converts $\widehat{K^{2}_1}$ into $\widehat{K^{2}_2}$. As a consequence there exists a finite sequence
of cubulated moves that carries $K^{2}_{1}$ onto ${K}^{2}_{2}$. 
\end{proof}

\subsection{$K^{2}_i$ is equivalent to $\widehat{K^{2}_i}$ by cubulated moves}\label{smooth}

\noindent Our goal is to prove Lemma A.

\begin{LA} 
Given a cubical 2-knot $K^2$, we can choose a small cubulation ${\cal{C}}_{m}$  fine enough so that 
$N^4_{K}=\{Q^4\in {\cal{C}}_{m}\,|\,Q^4\cap K^2\neq\emptyset\}$ is a closed tubular neighborhood of $K^2$ and
$Q^4\cap K^2$ is equal to either a vertex, one square, two squares sharing an edge (neighboring 2-faces) or three neighboring squares (two by two neighboring 2-faces). 
We can also choose $\widetilde{K^2}$ isotopic to $K^2$ such that $\widetilde{K^2}$ is 
${\cal{C}}^0$-arbitrarily close to $K^2$ and $\widetilde{K^2}\subset \mbox{Int}(N^4_{K})$. Let 
$\widehat{K^2}$ be an isotopic copy of $\widetilde{K^2}$ contained in $\partial N^4_{K}$, then $K^2\overset{c}\sim  \widehat{K^2}$; {\emph{i.e.}}, we can
go from $K^2$ to $\widehat{K^2}$ by a finite sequence of cubulated moves.
\end{LA}

\noindent \emph{Proof.} 
The knots $K^2$ and $\widehat{K^2}$ are isotopic and both are contained in the 2-skeleton of $N^4_{K}$; however $K^2\subset \mbox{Int}(N^4_{K})$  and $\widehat{K^2}\subset\partial N^4_{K}$. Our goal is to construct a connected 3-manifold ${\cal {M}}^3$  such that it will be contained in the 3-skeleton of $N^4_{K}$ and its boundary 
will consist of two connected components, namely $K^2$ and $\widehat{K^2}$.  

\noindent Let $B^4=\{Q^4\subset N^4_{K} \,|\,\,\,Q^4\cap \widehat{K^2}\neq\emptyset\}$.   

\noindent Since all hypercubes in $N^4_{K}$ intersect $K^2$, $B^4$ consists
of a finite number of hypercubes, say $m$, such that all of them also intersect $\widehat{K^2}$. 
By orienting $\widehat{K^2}$ we can enumerate the cubes in $B^4$ in such a way that
consecutive numbers belong to neighboring hypercubes (hypercubes sharing a common 3-face). 
To construct the 3-manifold ${\cal {M}}^3$, we will look at all possible cases of $Q^4_j\in B^4$ 
and find pieces $M_j^3$ which will be consist of the union of some  3-faces  of $Q^4_j$ such that they are two by two neighboring faces. The union of all $M_j^3$ will be ${\cal {M}}^3$. 
Notice that the boundary of these $M_j^3$'s will 
intersect both $K^2$ and $\widehat{K^2}$, hence 3-faces corresponding to neighboring 4-cubes will share a common face.   

\noindent In order to describe the cubes $M_j^3$'s in this proof, we will use the following notation for the 3-faces of  a given hypercube $Q^4_j$.
\begin{figure}[h] 
 \begin{center}
 \includegraphics[height=4.6cm]{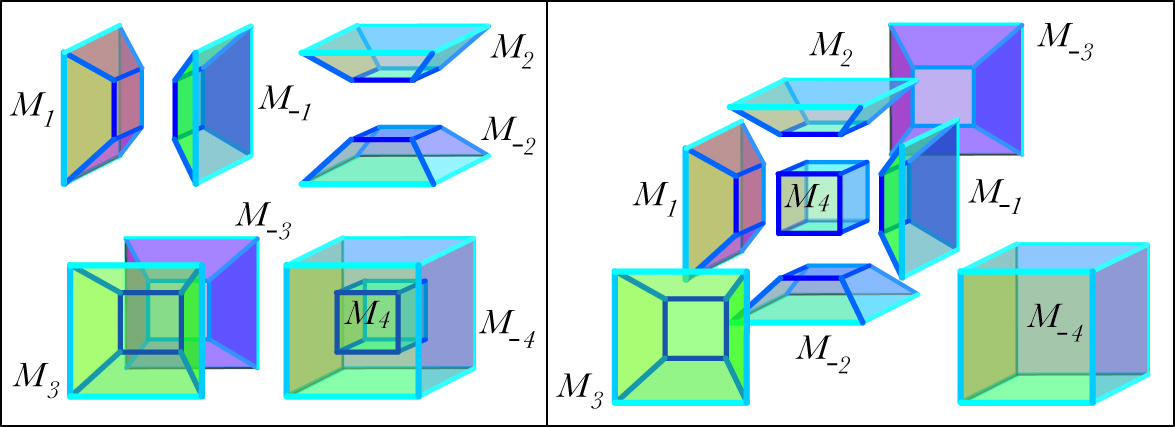}
 \caption{\sl Notation for the cubes in a hypercube $Q^4_j$.}
\end{center}
 
\label{not}
\end{figure}

\noindent Next, we will construct ${\cal {M}}^3$ considering
 all possible cases of both $K^2\cap Q^4_j$ and $\widehat{K^2}\cap Q^4_j$.  

\noindent {\bf Claim 1:} Let $Q^4_j\in B^4$ and suppose that $K^2\cap Q^4_j$ consists of  three squares (two by two neighboring 2-faces). 
Then, using face boundary moves if necessary, we can assume that 
$\widehat{K^2}\cap Q^4_j$ is a connected 2-disk consisting of the union of at most three squares.  

\noindent {\it Proof of Claim 1.} Suppose that $K^2\cap Q^4_j$ consists of three neighboring 2-faces $F_1^2$, $F_2^2$ and $F_3^2$. Observe that 
$F_1^2\cup F_2^2\cup F_3^2$ 
intersects fifteen distinct 2-faces of $Q^4_j$, hence $\widehat{K^2}\cap Q^4_j$ must be contained in the remaining six 2-faces. Notice that 
these six 2-faces are contained in a 3-face $C^3$ of $Q^4_j$ and applying  (M2)-moves if necessary, we can assume that at most three 2-faces of 
$\widehat{K^2}\cap Q^4_j$ lie on $C^3$. 
Therefore $K^2\cap Q^4_j$ consists of at most three 2-faces. See Figure \ref{3s}.1.
\begin{figure}[h] 
 \begin{center}
 \includegraphics[height=2.5cm]{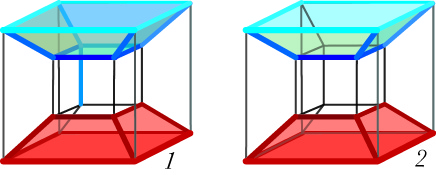}
\end{center}
\caption{\sl $K^2\cap Q^4_j$ consists of  three 2-faces.} 
\label{3s}
\end{figure}

\noindent This proves claim 1. $\square$

\noindent {\bf Case 1.} Suppose that $K^2\cap Q^4_j$ consists of three neighboring 2-faces $F_1^2$, $F_2^2$ and $F_3^2$. By the above claim, we have that
$\widehat{K^2}\cap Q^4_j$ is a connected 2-disk consisting of the union of at most three 2-faces; therefore $\widehat{K^2}\cap Q^4_j$ 
consists of at most three neighboring faces $E_1^2$, $E_2^2$ and $E_3^2$ such that $E_i^2$ is parallel
to $F_i^2$ ($i=1,2,3$).  Then the only possibility up to a face boundary moves $(M2)$ is shown in Figure \ref{3s}.2. Each $K^2\cap Q^4_j$  and  $\widehat{K^2}\cap Q^4_j$ consist of three neighboring 2-faces.
Thus  $M_j^3=M_3^3\cup M_{-1}^3\cup M_{-4}^3$.  

\noindent {\bf Claim 2:} \emph{Let $Q^4_j\in B^4$ and suppose that $K^2\cap Q^4_j$ consists of  two squares sharing an edge (neighboring 2-faces). Then, 
using face boundary moves if necessary, we can assume that $\widehat{K^2}\cap Q^4_j$ is a connected 2-disk consisting of the union of at most four faces.}  

\noindent {\it Proof of Claim 2.} Suppose that $K^2\cap Q^4_j$ consists of two neighboring 2-faces $F_1$ and $F_2$. Observe that the union 
$F_1\cup F_2$ intersects
 fifteen distinct 2-faces of $Q^4_j$, then $\widehat{K^2}\cap Q^4_j$ must be contained in the remaining seven 2-faces. 
 Notice that six of these seven 2-faces are contained in a 3-face $C^3$ of $Q^4_j$ and
applying  (M2)-moves if necessary, we can assume that three 2-faces of $\widehat{K^2}\cap Q^4_j$ lie on $C^3$. See Figure \ref{2s}.1.
Therefore $K^2\cap Q^4_j$ consists of at most four 2-faces. 

\begin{figure}[h] 
 \begin{center}
 \includegraphics[height=5cm]{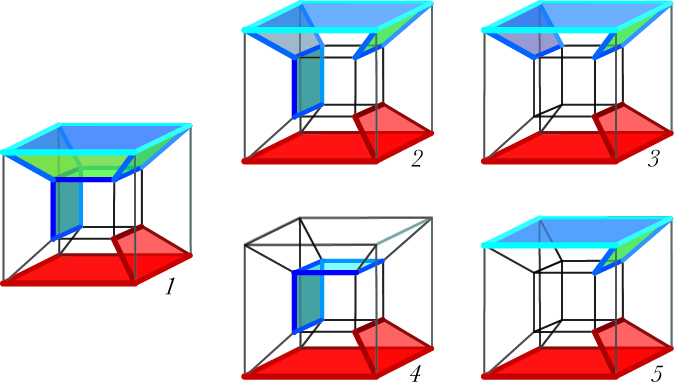}
\end{center}
\caption{\sl $K^2\cap Q^4_j$ consists of  two 2-faces.} 
\label{2s}
\end{figure} 
\noindent This proves claim 2. $\square$  
 
\noindent {\bf Case 2.} Suppose that $K^2\cap Q^4_j=F_1^2\cup F_2^2$ where $F_1^2$ and $F_2^2$ are two neighboring 2-faces. Since $\widehat{K^2}\cap Q^4_j$ is a connected 2-disk consisting of the union of at most four 2-faces, and $K^2\cap \widehat{K^2}=\emptyset$; we have that $\widehat{K^2}\cap Q^4_j$ is contained in the union of seven 2-faces
of $Q^4_j$, such that $F_1^2$, $F_2^2,\,\ldots,\,F_6^2$ belong to a 3-face $C^3$ and $F_7^2$  belongs to the opposite 3-face $\bar{C^3}$. See Figure \ref{2s}.1. Since $\widehat{K^2}$ intersects some  neighbor cubes of $Q^4_j$, we have three possibilities:
\begin{description}
\item [$(a)$] Suppose that $\widehat{K^2}\cap Q^4_j$ consists of the union of four squares.  Hence $F_7^2$ belongs to $\widehat{K^2}\cap Q^4_j$. Thus, there exist three 3-faces $M_1^3$, $M_{-1}^3$ and $M_{-4}^3$ that
contained two 2-faces of $(K ^2\cup \widehat{K^2})\cap Q^4_j$ (see Figure \ref{2s}.2). 
 Then $M_j^3=M_1^3\cup M_{-1}^3\cup M_{-4}^3$.

\item [$(b)$] Suppose that $\widehat{K^2}\cap Q^4_j$ consists of the union of three neighboring faces. Then considering all the possibilities 
satisfying that $K^2\cap Q^4_j$ can be extended, we have that  $\widehat{K^2}\cap Q^4_j$ consists of $F_1^2$, $F_2^2$ and $F_3^2$ such that 
$F_1^2$ and $F_3^2$ are disjoint,  $F_1^2$ and $F_2^2$ share an edge and so does
 $F_2^2$ and $F_3^2$ (see Figure \ref{2s}.3). So $M_j^3=M_1^3\cup M_{-1}^3\cup M_{-4}^3$.

\item [$(c)$] Suppose that $\widehat{K^2}\cap Q^4_j$ consists of the union of two neighboring faces. So considering all the options which $K^2\cap Q^4_j$ can be extended, we have two possibilities for $\widehat{K^2}\cap Q^4_j$: 
\begin{description}
\item [$(i)$] $\widehat{K^2}\cap Q^4_j$ is the union of the two 2-faces $F_7^2$ and $F_4^2$ (see Figure \ref{2s}.4). Then there exist two 3-faces $M_4^3$ and $M_{-2}^3$ containing two 2-faces of
$(K ^2\cup \widehat{K^2})\cap Q^4_j$. So $M_j^3=M_4^3\cup M_{-2}^3$.

\item [$(ii)$] $\widehat{K^2}\cap Q^4_j$ is the union of the two 2-faces  $F_2^2$ and $F_3^2$ (see Figure \ref{2s}.5). Then 
there exist two 3-faces $M_{-1}^3$ and $M_{-4}^3$ containing two 2-faces: one from $K ^2\cap Q^4_j$ and the other from  $\widehat{K^2}\cap Q^4_j$. 
So $M_j^3=M_{-1}^3\cup M_{-4}^3$.
\end{description}
\end{description}

\noindent {\bf Claim 3:} \emph{Let $Q^4_j\in B^4$ and suppose that $K^2\cap Q^4_j$ consists of one 2-face $F^2$. Then, using face boundary moves if necessary, we can assume that 
$\widehat{K^2}\cap Q^4_j$ is a connected 2-disk consisting of the union of at most five 2-faces.}  

\noindent {\it Proof of Claim 3.} Observe that $F^2$ intersects twelve distinct 2-faces of $Q^4_j$, then $\widehat{K^2}\cap Q^4_j$ must be contained
in the remaining eleven 2-faces. These eleven 2-faces are distributed in the following way: six of them are contained in a 3-face $C^3$ and the
remaining faces lie on a neighboring 3-face $D^3$ such that $C^3\cap D^3$ consists of one 2-face (see Figure \ref{1s}.1). Since $\widehat{K^2}\cap Q^4_j$ is a connected disk, 
then using (M2)-moves if necessary, we can assume that three 2-faces of $\widehat{K^2}\cap Q^4_j$ lie on $C^3$ and two 2-faces lie on $D^3$.
\begin{figure}[h] 
 \begin{center}
 \includegraphics[height=11cm]{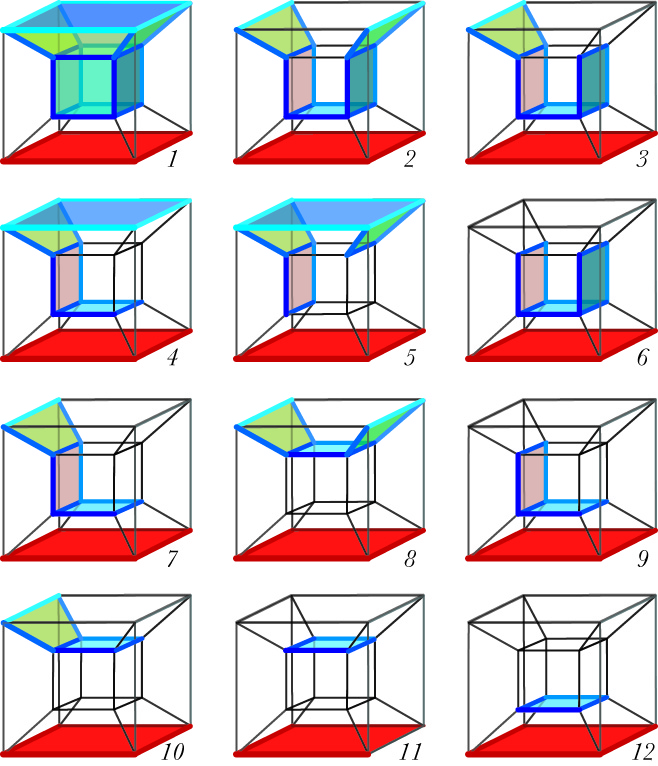}
\end{center}
\caption{\sl $K^2\cap Q^4_j$ consists of one 2-face.} 
\label{1s}
\end{figure} 
\noindent This proves claim 3. $\square$

 \noindent {\bf Case 3.} Suppose that $K^2\cap Q^4_j$ consists of one 2-face $F^2$. By claim 3,   we can assume that at most three 2-faces  
 $F_1^2$, $F_2^2$, $F_3^2$ of $\widehat{K^2}\cap Q^4_j$ lie on $C^3$ and two 2-faces $F_4^2$, $F_5^2$ lie on $D^3$. We have five possibilities:
\begin{description}
\item [$(a)$] Suppose that $\widehat{K^2}\cap Q^4_j$ consists of the union of five 2-faces. Then we have only one possibility for $\widehat{K^2}\cap Q^4_j$ 
(see Figure \ref{1s}.2).
There exist three 3-faces $M_1^3$, $M_{-2}^3$ and $M_{-1}^3$ which
contain two 2-faces of $(K ^2\cup \widehat{K^2})\cap Q^4_j$. Take $M_j^3=M_1^3\cup M_{-2}^3\cup M_{-1}^3$.

\item [$(b)$] Suppose that $\widehat{K^2}\cap Q^4_j$ consists of the union of four 2-faces. In this case, we have three possibilities:
\begin{description}
\item [$(i)$] $\widehat{K^2}\cap Q^4_j$ is the union of the four 2-faces $F_1^2$, $F_2^2$, $F_3^2$ and $F_4^2$. Then we have only one possibility (see Figure \ref{1s}.3). 
Let $M_{-2}^3$ be the 3-face of $Q^4_j$ such that $(K^2\cap Q^4_j)\cup F_2^2\subset M_{-2}^3$ and  let $M_1^3$ and $M_{-1}^3$ be the 3-faces of $Q^4_j$ such that each of them is a 
neighboring 3-face of $M_{-2}^3$ and  $F_1^2\subset M_1^3$ and $F_4^2\subset M_{-1}^3$. Then $M_j^3=M_1^3\cup M_{-1}^3\cup M_{-2}^3$.

\item [$(ii)$] $\widehat{K^2}\cap Q^4_j$ is the union of the 2-faces $F_1^2$, $F_2^2$ and $F_4^2$, $F_5^2$. Thus the only possible configuration is shown in 
Figure \ref{1s}.4. Let $M_1^3$, $M_{-2}^3$ and $M_{-1}^3$ be as above and let $M_{-4}^3$ be the neighboring 3-face such that $(K^2\cap Q^4_j)\cup F_5^2\subset M_{-4}^3$ 
so $M_j^3=M_1^3\cup M_{-2}^3\cup M_{-1}^3\cup M_{-4}^3$.

\item [$(iii)$] This configuration is similar to the one described in case $3b (i)$, but we exchange the cubes $C^3$ and $D^3$. Then, we get only one possibility 
(see Figure \ref{1s}.5 ). Let $M_j^3=M_1^3\cup M_{-1}^3\cup M_{-4}^3$.

\end{description}
\item [$(c)$] Suppose that $\widehat{K^2}\cap Q^4_j$ consists of the union of three 2-faces. Let $F_1^2$, $F_2^2$, $F_3^2$ be 2-faces contained in a
cube $C^3$ and let $F_4^2$, $F_5^2$, $F_6^2$ be 2-faces contained in the 3-cube $D^3$. Then we have three possibilities:
\begin{description}
\item [$(i)$] $\widehat{K^2}\cap Q^4_j$ is the union of the 2-faces $F_1^2$, $F_2^2$ and $F_3^2$.  
Let $M_1^3$, $M_{-1}^3$ and $M_{-2}^3$ be as above, so $M_j^3=M_1^3\cup M_{-1}^3\cup M_{-2}^3$ (see Figure \ref{1s}.6).

\item [$(ii)$] $\widehat{K^2}\cap Q^4_j$ is the union of the 2-faces $F_2^2$, $F_3^2$ and $F_6^2$. Thus $M_j^3=M_1^3\cup M_{-2}^3$
(see Figure \ref{1s}.7).

\item [$(iii)$] $\widehat{K^2}\cap Q^4_j$ is the union of the 2-faces $F_4^2$, $F_5^2$ and $F_6^2$. Let $M_{2}^3$ be the 3-face which contains the three squares $\widehat{K^2}\cap Q^4_j$. Then $M_j^3= M_2^3\cup M_{-4}^3 $ 
(see Figure \ref{1s}.8).

\end{description}
\item [$(d)$] Suppose that $\widehat{K^2}\cap Q^4_j$ consists of the union of two neighboring faces. Let $F_1^2$, $F_2^2$, $F_3^2$ be 2-faces contained in the
cube $C^3$ and let $F_4^2$, $F_5^2$, $F_6^2$ be 2-faces contained in the cube $D^3$. We have two possibilities:
\begin{description}
\item [$(i)$] $\widehat{K^2}\cap Q^4_j$ is the union of the 2-faces $F_1^2$ and $F_2^2$. Thus 
$M_j^3=M_1^3\cup M_{-2}^3$ 
(see Figure \ref{1s}.9).

\item [$(ii)$] $\widehat{K^2}\cap Q^4_j$ is the union of the 2-faces $F_3^2$ and $F_4^2$.  Thus 
$M_j^3=M_2^3\cup M_{-4}^3$ (see Figure \ref{1s}.10).

\end{description}
\item [$(e)$] Suppose that $\widehat{K^2}\cap Q^4_j$ consists of  one square. Then we have two possibilities:
\begin{description}
\item [$(i)$] $\widehat{K^2}\cap Q^4_j$ is the square $F_4^2$. Thus
$M_j^3=M_{-2}^3\cup M_4^3$
(see Figure \ref{1s}.11).

\item [$(ii)$] $\widehat{K^2}\cap Q^4_j$  is the 2-face $F_2^2$. Thus
$M_j^3=M_{-2}^3$
(see Figure \ref{1s}.12).

\end{description}

\end{description}

\noindent {\bf Claim 4:} \emph{Let $Q^4_j\in B^4$ and suppose that $K^2\cap Q^4_j$ consists of   an edge $e$. Then, using face boundary moves if necessary, we can assume that 
$\widehat{K^2}\cap Q^4_j$ is a connected 2-disk consisting of the union of at most six faces.}  

\noindent {\it Proof of Claim 4.} Suppose that $K^2\cap Q^4_j$ consists of an edge $e$. See Figure \ref{1e}. Then $\widehat{K^2}$ must turn on $Q^4_j$; in other words,  
$\widehat{K^2}\cap Q^4_j$ must contain two
faces $F_1^2$ and $F_2^2$ such that $F_1^2=v+\{ae_{i_1}+be_j\,:\,0\leq a,\,b\leq 1\}$ and $F_2^2=v+\{ce_{i_2}+de_j\,:\,0\leq c,\,d\leq 1\}$, where $e_{i_1}$, $e_{i_2}$ and
$e_{j}$ are distinct canonical vectors and $v$ is a vector with integer coordinates. Notice that $e_j$ is parallel to the edge $e$ and only 
nine 2-faces of $Q^4_j$ satisfy both they do not intersect $K^2$ and they have an edge parallel to $e_j$; 
hence $\widehat{K^2}\cap Q^4_j$ must be contained in the union of these nine 2-faces. Since
$\widehat{K^2}\cap Q^4_j$ is a connected disk it follows,  up to applying face boundary moves (M2) if necessary, that it consists of the union of at most six faces. 
\noindent This proves claim 4. $\square$  

\begin{figure}[h] 
 \begin{center}
 \includegraphics[height=11cm]{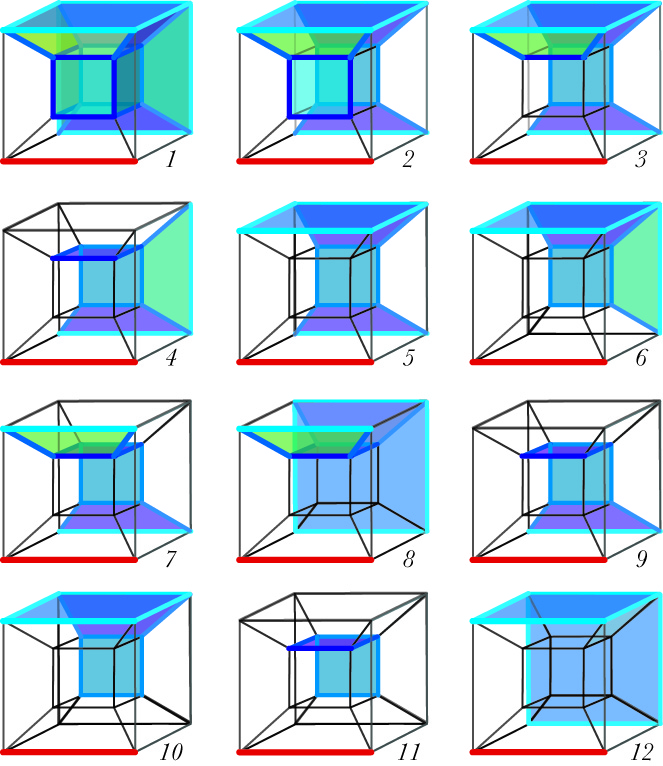}
\end{center}
\caption{\sl $K^2\cap Q^4_j$  is an edge.} 
\label{1e}
\end{figure}

 \noindent {\bf Case 4.} Suppose that $K^2\cap Q^4_j$ consists of an edge $e$. By claim 4, 
$\widehat{K^2}\cap Q^4_j$ is contained in three 3-faces $M_3^3$, $M_{-2}^3$ and $M_{-4}^3$ such that any two of them share a  2-face.
\begin{description}
\item [$(a)$] $\widehat{K^2}\cap Q^4_j$ consists of the union of six 2-faces. Then, applying  (M2)-moves if necessary, we have that $\widehat{K^2}\cap Q^4_j$ consists of the 
union of two 2-faces contained in $M_3^3$, three 2-faces  in $M_{-3}^3$ and one 2-face in $M_{-4}^3$ (see Figure \ref{1e}.2). Thus $M_j^3=M_3^3\cup M_{-3}^3\cup M_{-4}^3$.

\item [$(b)$] $\widehat{K^2}\cap Q^4_j$ consists of the union of five 2-faces. Then, applying  (M2)-moves if necessary, as in the previous case
we have that $M_j^3=M_3^3\cup M_{-3}^3\cup M_{-4}^3$. (see Figure \ref{1e}.3). 

\item [$(c)$] $\widehat{K^2}\cap Q^4_j$ consists of the union of four 2-faces. We have five possibilities:
\begin{description}
\item [$(i)$] $\widehat{K^2}\cap Q^4_j$ consists of two 2-faces contained in $M_4^3$ and three 2-faces  in $M_{-3}^3$. 
Thus $M_j^3=M_{-3}^3\cup M_2^3\cup M_{-4}^3$ (see Figure \ref{1e}.4).

\item [$(ii)$] $\widehat{K^2}\cap Q^4_j$ consists of one 2-face contained in $M_{-4}^3$ and three 2-faces  in $M_{-3}^3$. 
 Then $M_j^3=M_{-3}^3\cup M_{-4}^3$ (see Figure \ref{1e}.5).

\item [$(iii)$] $\widehat{K^2}\cap Q^4_j$ consists of one square contained in $M_{-4}^3$ and three 2-faces  in $M_{-3}^3$ (see Figure \ref{1e}.6). Then $M_j^3=M_{-3}^3\cup M_{-4}^3$.

\item [$(iv)$] $\widehat{K^2}\cap Q^4_j$ consists of two 2-faces contained in $M_{-3}^3$ and two 2-faces  in $M_2^3$ (see Figure \ref{1e}.7). 
Then $M_j^3=M_{-3}^3\cup M_{2}^3\cup M_{-4}^3$.

\item [$(v)$] $\widehat{K^2}\cap Q^4_j$ consists of three 2-faces contained in $M_2^3$ and one square in $M_{-4}^3$ (see Figure \ref{1e}.8). 
So $M_j^3= M_2^3\cup M_{-4}^3$.

\end{description}

\item [$(d)$] $\widehat{K^2}\cap Q^4_j$ consists of the union of three 2-faces. We have two possibilities:
\begin{description}
\item [$(i)$] $\widehat{K^2}\cap Q^4_j$ consists of one square contained in $M_2^3$ and two 2-faces  in $M_{-3}^3$ (see Figure \ref{1e}.9). 
 Then $M_j^3=M_2^3\cup M_{-3}^3\cup M_{-4}^3$.
\item [$(ii)$] $\widehat{K^2}\cap Q^4_j$ consists of two 2-faces contained in $M_{-3}^3$ and one 2-face  in $M_{-4}^3$ (see Figure \ref{1e}.10). 
 Then $M_j^3=M_{-3}^3\cup M_{-4}^3$.

\end{description}

\item [$(e)$] $\widehat{K^2}\cap Q^4_j$ consists of the union of two neighboring 2-faces. We have two possibilities:
\begin{description}
\item [$(i)$] $\widehat{K^2}\cap Q^4_j$ consists of two 2-faces contained in $M_4^3$  (see Figure \ref{1e}.11). 
 Then $M_j^3=M_4^3\cup M_3^3$.

\item [$(ii)$] $\widehat{K^2}\cap Q^4_j$ consists of two 2-faces contained in $M_{-4}^3$  (see Figure \ref{1e}.12). 
Thus $M_j^3=M_{-4}^3$.

\end{description}
\end{description}

\noindent {\bf Claim 5:} \emph{Let $Q^4_j\in B^4$ and suppose that $K^2\cap Q^4_j$ consists of  a vertex $v$. Then, using face boundary moves if 
necessary, we can assume that 
$\widehat{K^2}\cap Q^4_j$ is a connected 2-disk consisting of the union of at most six 2-faces.}  

\noindent {\it Proof of Claim 5.} 
 Since $K^2\cap Q^4_j$ consists of a vertex $v$, then $v$ must be a \emph{corner} of $K^2$ \emph{i.e.} $v$ is a common vertex of three 
 neighboring faces of 
$K^2$.  Hence $\widehat{K^2}$ must have a corner at some vertex of $Q^4_j$. Since $K^2\cap\widehat{K^2} =\emptyset$,
it follows that $\widehat{K^2}\cap Q^4_j$ can be contained in either one, two or three 3-faces of $Q^4_j$ such that
these 3-faces do not contained $v$. Suppose that $\widehat{K^2}\cap Q^4_j$ is a connected disk consisting of at least six 2-faces, 
then by a combinatorial analysis considering all the possible
descriptions of  $\widehat{K^2}$, we should have that
four of these 2-faces are contained in some 3-face of $Q^4_j$; hence applying a face boundary move (M2), we get that 
$\widehat{K^2}\cap Q^4_j$ can be reduced to at most six 2-faces. (see Figure \ref{1v}.1.)
\begin{figure}[h] 
 \begin{center}
 \includegraphics[height=5cm]{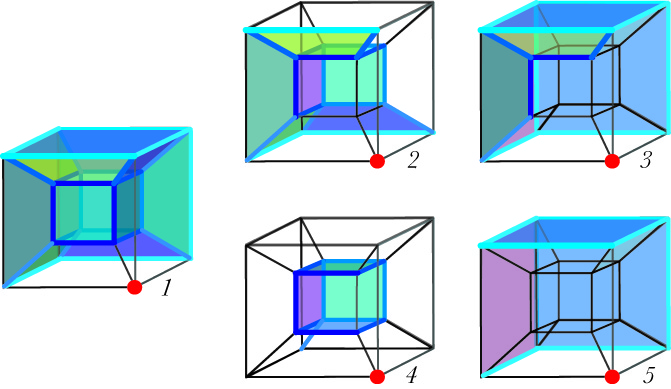}
\end{center}
\caption{\sl $\widehat{K^2}\cap Q^4_j$ consists of a vertex.} 
\label{1v}
\end{figure}  This proves claim 5. $\square$  

\noindent {\bf Case 5.} Suppose that $K^2\cap Q^4_j$ consists of a vertex $v$. By the above claim,
 $\widehat{K^2}\cap Q^4_j$ is the union of at least two 2-faces of $Q^4_j$ and its boundary must be
contained in the union of the three neighboring 3-faces containing $v$.
\begin{description}
\item [$(a)$] $\widehat{K^2}\cap Q^4_j$ is the union of six faces. Then $\widehat{K^2}\cap Q^4_j$ 
is contained in three 3-faces $M_{-2}^3$, $M_{3}^3$ and $M_{4}^3$ of $Q^4_j$ (see Figure \ref{1v}.2). 
Since  $\widehat{K^2}\cap Q^4_j$ is homeomorphic to a disk, we have that 
each 3-face  $M_{i}^3$ must contain two 2-faces of it and two of these three 3-faces contain $v$. 
Then $M_j^3=M_{-2}^3 \cup M_{3}^3\cup M_{4}^3$.

\item [$(b)$] $\widehat{K^2}\cap Q^4_j$ is the union of five 2-faces. Then $\widehat{K^2}\cap Q^4_j$ 
is contained in the union of two 3-faces $M_{-4}^3$ and  $M_{3}^3$ of $Q^4_j$ (see Figure \ref{1v}.3). So $M_j^3= M_3^3\cup M_{-4}^3$.

\item [$(c)$] $\widehat{K^2}\cap Q^4_j$ is the union of four 2-faces. Again, $\widehat{K^2}\cap Q^4_j$ 
is contained in a 3-face $M_{4}^3$ of $Q^4_j$ and its boundary must be
contained in the union of the three neighboring 3-faces containing $v$. Hence, one 3-face $M_{3}^3$, $M_{-1}^3$, $M_{-2}$ must contain one 2-face.  The only possibility is that four 2-faces lie in some 3-face $M_{4}^3$ of  $Q^4_j$ (see Figure \ref{1v}.4). 
 In this case $M_j^3=M_{4}^3$.

\item [$(d)$] $\widehat{K^2}\cap Q^4_j$ is the union of three neighboring 2-faces. These three 2-faces must be contained in some 3-face $M_{-4}^3$. 
Then $M_j^3=M_{-4}^3$ 
(see Figure \ref{1v}.5).

\item [$(e)$] $\widehat{K^2}\cap Q^4_j$ is the union of two neighboring 2-faces. These two 2-faces must be contained in some 3-face $M_{-4}^3$. 
Then $M_j^3=M_{-4}^3$.
\end{description}

\noindent Next, we will construct our 3-manifold ${\cal {M}}^3$.  

\noindent Let  ${\cal {M}}^3=\cup{M_j^3}$. By construction, each $M_j^3$ consists of the union of 3-faces contained in the  3-skeleton of ${N}_K^4$,
hence ${\cal{M}}^3$ is contained in the  3-skeleton of ${{N}}_K^4$.

\noindent As we mention before,  if we apply an (M2)-move on a hypercube $Q^4_j$, then this movement does not affect the corresponding choice of $M_{l}^3$ on its neighbor hypercube $Q^4_{l}$; in other words,
the choice of $M_j^3$ depends only of the configuration of $\widehat{K^2}\cap Q^4_j$ and $K^2\cap Q^4_j$ in $Q^4_j$. Observe that the boundary of each $M_j^3$ is composed by  2-faces belonging to  $K^2$ and $\widehat{K^2}$  and 
some disjoint faces $F^j_i$ ($i=1,2,\ldots ,s_j$) that do not belong to
neither $K^2$ nor $\widehat{K^2}$. Thus, if we take a neighbor hypercube of $Q^4_j$, say $Q^4_{j_{l}}\in B^4$, and we construct $M_j^3$ and $M_{j_{l}}^3$, then 
the intersection $M_j^3\cap M_{j_l}^3$ consists of the union of some $F^{j}_{r_i}$; notice that each of these 2-faces $F^j_i$ belongs to a neighbor hypercube of $Q^4_j$, 
$Q^4_{j_{l_s}}\in B$. Hence ${\cal{M}}^3$ is a 3-manifold whose boundary
consists of two connected components, namely $K^2$ and $\widehat{K^2}$.

\noindent Now, we will carry the knot $\widehat{K^2}$ onto the knot $K^2$ via a finite number of cubulated moves. Notice that ${\cal{M}}^3$ is 
the union 
of $m$ components $M_1^3,\,\ldots M_m^3$ which are enumerated in such a way that   $M_{n+1}^3$ is a 
neighbor of $M_n^3$ for all $n$, \emph{i.e.}, $M_{n+1}^3\cap M_n^3$ consists of a 2-face $F_n^2$.  

\noindent We will use induction on $m$. Consider $M_1^3$. We apply  (M2)-moves on $M_1^3$ in such a way that the faces (of any dimension)
belonging to $K^2$ are replaced by 
those belonging to $\widehat{K^2}$ (by construction $M_1^3\cap K^2\neq\emptyset$ and $M_1^3\cap \widehat{K^2}\neq \emptyset$).
Next, we consider $M_2^3$. By the previous step, $M_1^3$ and $M_2^3$ share 2-faces belonging to $\widehat{K^2}$. Then we apply again (M2)-moves, 
the faces (of any dimension) belonging to $K^2$ are replaced by those belonging to $\widehat{K^2}$. We continue this finite process in this way. 
Notice that if $l\subset  K^2$  
then $l\subset \partial {\cal{M}}^3$, thus $l$ is not a common 2-face of some pair $M_i^3$, $M_j^3$ belonging to ${\cal{M}}^3$, hence if 
$l$ is replaced by a 2-face belonging to $\widehat{K^2}$, then this replacement will be kept in the following steps. Therefore, the result follows. $\square$

\subsection{$\widehat{K^2_1}$ is equivalent to $\widehat{K^2_2}$ via cubulated moves}\label{cubic}

In this section, we shall prove the Lemma B which says that  $\widehat{K^2_1}$ is equivalent to $\widehat{K^2_2}$ by cubulated moves.  

\noi Consider again the projection $p:\mathbb{R}^{5}\hookrightarrow\mathbb{R}$ on the last coordinate. We call \emph{horizontal hyperplane} to an affine hyperplane
parallel to the space $\mathbb{R}^{4}\times\{0\}$. 
Thus $p^{-1}(t)=\mathbb{R}^{4}_t$ is a horizontal hyperplane. Observe that each hyperplane 
$\mathbb{R}^{4}_t$ has a canonical cubulation
given by the restriction of the canonical cubulation of $\mathbb{R}^{5}$ to it.  

\begin{definition}
 A $2$-cell ($2$-face) $F^2$ of the canonical cubulation $\cal{C}$ of $\mathbb{R}^{5}$  is called \textit{horizontal} if $p(F^2)$ is 
 a constant number in $\mathbb{N}$.
A $2$-cell $F^2$ is \textit{vertical} if $p(F^2)=[m,m+1]$ for some $m\in\mathbb{N}$.
\end{definition}

\begin{definition}
 Let $\Sigma^3$ be a cubulated $3$-manifold and $P^4$ be a horizontal hyperplane in $\mathbb{R}^{5}$.
We say that $P^4$  \emph{intersects transversally to} $\Sigma^3$, denoted by  $P^4\pitchfork\Sigma^3$, if $P^4$ intersects transversally each $k$-cube of $\Sigma^3$, $k\geq 1$.
\end{definition}

\begin{lem}
 Let $\mathbb{R}^{4}_t \subset \mathbb{R}^5$, $t\not\in\mathbb{Z}$ be an affine hyperplane. Then  $\mathbb{R}^{4}_t$ intersects $\widehat{J^3}$ transversally.
\end{lem}

\noindent{\it Proof.} By Theorem \ref{trace},  
$\mathbb{R}^{4}_t\cap \widehat{J^3}$ is connected. Let $x\in \mathbb\mathbb{R}^{4}_t\cap \widehat{J^3}$. Then
$x\in Q^3_i$, where $Q^3_i$ is a 3-face of the cubulation of $\mathbb{R}^{5}$. Notice that $Q^3_i$ is a vertical $3$-face, since $t\not\in\mathbb{Z}$. 
So, we have two possibilities:
either $x\in\mbox{Int}(Q^3_i)$ or  $x\in F^2\setminus\partial F^2$ where $F^2$ is a 2-face of $Q^3_i$, or $x$ belongs to an edge.
\begin{enumerate}
\item If $x\in\mbox{Int}(Q^3_i)$ then $\mathbb{R}^{4}_t\cap Q^3_i=E^2$, where $E^2$ is a square parallel to a 2-face and $x\in E^2$.
\item If $x$ belongs to $F^2\setminus\partial F^2$, then there exists another vertical $3$-face $Q^3_j$ such that $x\in Q^3_i\cap Q^3_j$. Thus
$\mathbb{R}^{4}_t\cap Q^3_i$ must be a square $E_i^2$ such that it is parallel to a 2-face,  and analogously $\mathbb{R}^{4}_t\cap Q^3_j$ is also a square $E_j^2$,
thus $x\in E_i^2\cap E_j^2$. 
\item If $x$ belongs to an edge $l$, then there exist  vertical $3$-faces $Q^3_{i_r}$ $r=1,2,3$ such that $x\in Q^3_j\cap_{r=1}^3 Q^3_{i_r}$. As in the previous case,
$\mathbb{R}^{4}_t\cap Q^3_{i_r}$ must be a square $E_{i_{r}}^2$ parallel to a 2-face,  and analogously $\mathbb{R}^{4}_t\cap Q^3_j$ is also a square $E_j^2$, hence
 $x\in E_j^2\cap_{r=1}^4 E_{i_r}^2$. 
\end{enumerate}

\noindent Therefore, the result follows. $\square$  

\begin{coro}
For $x\not\in\mathbb{Z}$, the set $p^{-1}(x)\cap \widehat{J^3}$ is a $2$-knot.
\end{coro}

\noindent{\it Proof.} By the above, $p^{-1}(x)\cap \widehat{J^3}$ is a cubulated compact connected surface. Since
$\widehat{J^3}$ is homeomorphic to $\mathbb{S}^2\times\mathbb{R}$, it follows that $p^{-1}(x)\cap \widehat{J^3}$ is homeomorphic to $\mathbb{S}^2$. $\square$  

\noindent Now, for each $n\in\mathbb{N}$ we define 
$$
K^2_{n-}:= p^{-1}\left(n-\frac{1}{2}\right)\cap \widehat{J^3}
$$ 
and 
$$
K^2_{n+}:= p^{-1}\left(n+\frac{1}{2}\right)\cap \widehat{J^3}.
$$
Observe that $K^2_{n-}$ and $K^2_{n+}$ are cubical $2$-knots.

\noi Let ${\cal{C}}^3$ be the set of 3-cubes ($3$-cells) belonging to the canonical cubulation of $\mathbb{R}^{5}$. Consider the three spaces:
$$
M^3_{n-}:=\{Q^3\in \mathcal{C} ^3 \,|\,Q^3 \cap K^2_{n-}\neq\emptyset\}, 
$$

$$
M^3_{n+}:=\{Q^3\in \mathcal{C} ^3 \,|\,Q^3 \cap K^2_{n+}\neq\emptyset\} 
$$
and  $M^3_n:= p^{-1}(n)\cap \widehat{J^3}$. 

\noi By construction $M^3_{n-}= K^2_{n-}\times [0,1]$
and $M^3_{n+}= K^2_{n+}\times [0,1]$.  

\noindent Let $M^3:=M^3_{n-}\cup M^3_{n}\cup M^3_{n+}$. Hence
$M^3=\mbox{Cl}(p^{-1}(n-1,n+1)\cap \widehat{J^3})$, where Cl denotes closure.

\begin{lem}\label{circle}
 The space $M^3$ is homeomorphic to $\mathbb{S}^{2}\times [0,1]$.
\end{lem}

\noindent{\it Proof.} Since $M^3$ is a compact submanifold of $\widehat{J^3}$, then by Lemma \ref{trace}, it is also connected. Now $\widehat{J^3}$ is
homeomorphic to $\mathbb{S}^{2}\times\mathbb{R}$, and $\widehat{J^3}-M^3$ has two connected components, hence the result follows. $\square$ 

\begin{lem}
 $M^3_n$ has the homotopy type of $\mathbb{S}^{2}$.
\end{lem}

\noindent{\it Proof.} Consider the set $M^3$. Then by the previous Lemma, we have that  $M^3\cong \mathbb{S}^{2}\times [0,1]$. Notice that
$K^2_{n-}\times\{0\}\cong \mathbb{S}^2\times\{0\}$ and $K^2_{n+}\times\{0\}\cong \mathbb{S}^2\times\{1\}$. 
Hence $\widetilde{M^3}=M^3/(K^2_{n-}\times\{0\})\cup (K^2_{n+}\times\{1\})$ is homeomorphic to $\mathbb{S}^3$. 
By Alexander duality,
we have that $\widetilde{H}_0(\widetilde{M^3}-M^3_n,\mathbb{Z})=\widetilde{H}^{2} (M_n^3,\mathbb{Z})$, but $\widetilde{M^3}-M_n^3$ has two simply connected components, so
$\widetilde{H}_{0} (\widetilde{M^3},\mathbb{Z})\cong \mathbb{Z}$. Since $M_n^3\subset M^3\cong \mathbb{S}^{2}\times [0,1]$, we have that either
$\pi_2 (M_n^3)\cong \{0\}$ or  $\pi_2 (M_n^3)\cong \mathbb{Z}$, but $\tilde{H}^{2}(M_n^3,\mathbb{Z})\cong\mathbb{Z}$, hence $\pi_2 (M_n^3)\cong \mathbb{Z}$. 

\noi Again by Alexander duality, we have that $\widetilde{H}_1(\widetilde{M^3}-M_n^3,\mathbb{Z})=\widetilde{H}^{1} (M_n^3,\mathbb{Z})$, but
$\widetilde{H}_1(\widetilde{M^3}-M_n^3,\mathbb{Z})\cong \{0\}$ hence $\widetilde{H}^{1} (M_n^3,\mathbb{Z})\cong \{0\}$. This implies that 
$\pi_1 (M_n^3)\cong \{0\}$.
Therefore, $M_n^3$ has the homotopy type of $\mathbb{S}^2$. $\square$  

\begin{lem}
The space $M^3$ retracts strongly to $M_n^3$.
\end{lem}

\noindent{\it Proof.} Since $M^3=M^3_{n-}\cup M^3_{n}\cup M^3_{n+}$ is homeomorphic to $\mathbb{S}^2\times [0,1]$, 
and $M^3_{n-}= K^2_{n-}\times [0,1]$ and $M^3_{n+}= K^2_{n+}\times [0,1]$, 
we have that 
$M^3_{n-}= K^2_{n-}\times [0,1]$ retracts strongly to $K^2_{n-}\times \{1\}$ and
$M^3_{n+}= K^2_{n+}\times [0,1]$ retracts strongly to $K ^2_{n+}\times \{0\}$. Now
$\partial M_n^3=(K^2_{n-}\times\{1\})\cup (K^2_{n+}\times\{0\})$. Therefore, the result follows. $\square$

\noindent Next, we are going to describe the subset $M_n^3$. Notice that the squares of $M_n^3$ are of four types, which we will denote 
by $T_{-}$, $T_{+}$, $T_{\pm}$ and $T$.

\begin{itemize}
\item A  square $F^2\subset M_n^3$ belongs to $T_{-}$ if  $F^2\subset M^3_{n-}$ but $F^2\not\subset M^3_{n+}$.

\item A square $F^2\subset M_n^3$ belongs to $T_{+}$ if $F^2\subset  M^3_{n+}$ but $F^2\not\subset M^3_{n-}$.

\item A square $F^2\subset M_n^3$ belongs to $T_{\pm}$ if $F^2\subset  M^3_{n-}\cap M^3_{n+}$.

\item A square $F^2\subset M_n^3$ belongs to $T$ if $F^2\not\subset  M^3_{n+}\cup M^3_{n-}$.
\end{itemize}

\noi By the above Lemma, there are copies of $K^2_{n-}$ and $K^2_{n+}$ contained in $\partial M_n^3$. By abuse of notation we will
denote them in the same way. Notice that $K^2 _{n-}$ is the union of squares of types $T_{-}$ and $T_{\pm}$, and
$K^2_{n+}$ is the union of squares of types $T_{+}$ and $T_{\pm}$.

\begin{lem}\label{fns}
$K^2_{n-}\overset{c}\sim  K^2_{n+}:$ There exists a finite sequence of cubulated moves that carries the  2-knot $K^2_{n-}$ into the 2-knot $K^2_{n+}$.
\end{lem}

\noi {\it Proof.} We will show it by cases.  
\noi Case 1. Suppose that $K^2_{n-}=K^2_{n+}$.  Clearly, the result is true.  

\noi Case 2. Suppose that $K^2_{n-}\cap K^2_{n+}=\emptyset$. In other words, $K^2_{n-}$ and 
$K^2_{n+}$
 do not have squares of type $T_{\pm}$.   
Remember that $M_n^3$ is a cubical compact 3-manifold whose fundamental group is isomorphic to $\mathbb{Z}$. Thus  $\partial M_n^3$
has two connected components; namely $K^2_{n-}$ and $K^2_{n+}$, such that their intersection is empty.
Hence $M_n^3$ is the union of a finite number of cubes (3-faces) belonging to the 3-skeleton of the cubulation $\cal C$, whose 2-faces are of any of the types $T_{-}$, $T_{+}$ and $T$.  

\noi Next, we will carry the $2$-knot $K^2_{n-}$ onto the $2$-knot $K^2_{n+}$ via a finite number of cubulated moves; {\it i.e.} we will
carry the squares of type $T_{-}$ onto the squares of type $T_{+}$. Let $Q^3$ be a cube contained in $M_n^3$. We can assume, up to $(M1)$-move, that if a square
$F^2\subset Q^3$ belongs to $T_{-}$, then $Q^3\cap K^2_{n+}=\emptyset$  and
$Q^3\cap K^2_{n-}$  consists of either a square, two neighboring 2-faces or three neighboring 2-faces. Analogously, if
 $F^2\subset Q^3$ belongs to $T_{+}$, then $Q^3\cap K^2_{n-}=\emptyset$ and
$Q^3\cap K^2_{n+}$ consists of a square, two neighboring 2-faces or three neighboring 2-faces.  

\noi The compact space $M_n^3$ is the union of a finite number of cubes, say $m$. We will enumerate them by levels in the following way. Remember that $M_n^3$ is a cubical compact 3-manifold, so we may assume that $M_n^3\subset\mathbb{R}_+^3$. 
Let $q:\mathbb{R}^3\hookrightarrow \mathbb{R}$ be the projection
on the last coordinate. We define  the level $k$, $M_k^3:=M_n^3\cap q^{-1}([k,k+1])$, for $k\in\mathbb{N}$. Notice that  there exists $k_1$ and $k_2$ positive integers such that $M_k^3=\emptyset$ for $k_1\leq k\leq k_2$. Then we start enumerating the cubes $Q^3\in M^3$ by levels. We start at the level $k_1$. The first cube
$Q_1^3$ contains a 2-face of type $T_{-}$, and given the cube $Q_n^3$, the cube  $Q_{n+1}^3$ shares a 2-face $F_n^2$ with  $Q_n^3$ and whenever 
it is possible, we choose $Q_{n+1}^3$ in such a way that $F_n^2$ is parallel to $F_{n-1}^2$; otherwise we choose  $Q_{n+2}^3$ such that $F^2_{n+1}$ is parallel
to $F^2_{n-1}$, at the end we continue on the level $k_1+1$ and so on.  

\noi We will use induction on $m$. Consider the cube $Q_1^3$. We apply the $(M2)$-move to $Q_1^3$ replacing the 2-faces of type $T$ by 2-faces of type  $T_{-}$.
We consider $Q_2^3$. Observe that $Q_1^3$ and $Q_2^3$ share a 2-face of type $T_{-}$. Then we apply again the $(M2)$-move replacing the 2-faces of type 
$T$ by 2-faces of type  $T_{-}$. We continue inductively. 

\noi Notice that if $F^2\subset  M_n^3$ is a 2-face of type $T_{+}$ then $F^2\subset \partial M_n^3$;
so $F^2$ is not a common 2-face of two cubes $Q_i^3$ and $Q^3_j$ in $M_n^3$; hence if  $F^2$ is replaced by a 2-face of type $T_{-}$ then this replacement is not modified in 
any other next step. Therefore, the result follows.  

\noi Case 3. Suppose that the intersection $K^2_{n-}\cap K^2_{n+}$ contains a finite number of 2-faces.  
\noi The 3-manifold $M_n^3$ consists of connected components $C_i^3$, $i=1,\ldots,r$ such that each $C_i^3$ is the union of cubes
$Q_{i_1}^3, \ldots, Q_{i_{m_i}}^3\in{\cal{C}}$ and the intersection $C_i^3\cap C_j^3$ is either empty or a square belonging to 
$K^2_{n-}\cap K^2_{n+}$. Therefore, we apply the previous argument to each $C_i^3$.

\noi Case 4. Suppose that the intersection $K^2_{n-}\cap K^2_{n+}$ contains a square of type $T_{\pm}$. The 3-manifold $M_n^3$ consists of 3-dimensional connected components $C_i^3$, $i=1,\ldots,r$
 and cubical 2-disks $\gamma_{ij}$. As before, each $C_i^3$ is a union of cubes $Q_{i_1}^3, \ldots, Q_{i_{m_i}}^3\in{\cal{C}}$, and $\gamma_{ij}$ is a 
 cubical disk (or edge) joining the component $C_i^3$ with the component $C_j^3$.
Observe that if $\gamma_{ij}$ is the union of 2-faces of type $T_{\pm}$, hence $\gamma_{ij}\subset K^2_{n-}\cap K^2_{n+}$.
Moreover $K^2_{n-}\cap K^2_{n+}= \gamma_{ij}$ and  $\partial M_n^3=K^2_{n-}\cup K^2_{n+}$.  

\noi Since $\pi_2 (M_n^3)\cong\mathbb{Z}$, then  $C_i^3$  the homotopy type either the 2-sphere or the 3-ball. 
Suppose that $C_i^3$ has the homotopy type of the 2-sphere,
then by hypothesis $\partial C_i^3=\partial M_n^3=K^2_{n-}\cup K^2_{n+}$ contains a face $F^2$ of type $T_{\pm}$, but $F^2$ does not belong to
any cube $Q^3$ of $M_n^3$; so $F^2$ does not belong to $C_i^3$. This is a contradiction, hence  $C_i^3$ has the homotopy type of the 3-ball.  

\noi By the above, $\partial C_i^3$ is homeomorphic to $\mathbb{S}^2$ and consists of 2-faces of type $T_{-}$ and $T_{+}$. Moreover, $\partial C_i^3$ consists of two disks
$D_-^2$ and $D_+^2$, such that $D_-^2$ is the union of faces of type $T_{-}$ and $D_+^2$ is the union of faces of type $T_{+}$. Now we apply the
argument of the case 2, so $D_+^2$ is replaced by $D_-^2$. Since we have a finite number of components $C_i^3$,
the result follows. $\square$

\begin{LB} 
$\widehat{K^2_1}\overset{c}\sim  \widehat{K^2_2}$: There exists a finite sequence
of cubulated moves that carries $\widehat{K^2_1}$ into  $\widehat{K^2_2}$. In other words, $\widehat{K^2_1}$ is equivalent to  $\widehat{K^2_2}$ by cubulated moves. \end{LB}

\noindent{\it Proof of Lemma B.}  Recall that there exist integer numbers 
$m_1$ and $m_2$ such that $p^{-1}(t)\cap \widehat{J^3} = \widehat{K^2_1}$ for all $t\leq m_1$ and
$p^{-1}(t)\cap \widehat{J^3} = \widehat{K^2_2}$ for all $t\geq m_2$. Consider the integer $m_1+1$. By 
Lemma \ref{fns} there exists a finite number of cubulated moves that
carries the 2-knot $\widehat{K^2_1}$ into the 2-knot $K^2_{(m_1+1)-}$. We continue inductively, and again by 
Lemma \ref{fns}
there exists a finite number of cubulated moves that
carries the knot $K^2_{(m_2-1)+}$ into the knot $\widehat{K^2_2}$. Since, we have a finite number of integers contained in the interval $[m_1,m_2]$, then
there exists a finite sequence of cubulated moves that carries $\widehat{K^2_1}$ into  $\widehat{K^2_2}$. $\square$

\noindent J. P. D\'iaz. {\tt Instituto de Matem\'aticas, Unidad Cuernavaca}. Universidad Nacional Au\-t\'o\-no\-ma de M\'exico.
Av. Universidad s/n, Col. Lomas de Chamilpa. Cuernavaca, Morelos, M\'exico, 62209.

\noindent {\it E-mail address:} juanpablo@matcuer.unam.mx
\vskip .3cm
\noindent G. Hinojosa. {\tt Centro de Investigaci\'on en Ciencias}. Instituto de Investigaci\'on en Ciencias B\'asicas y Aplicadas. Universidad Aut\'onoma del Estado de Morelos. Av. Universidad 1001, Col. Chamilpa.
Cuernavaca, Morelos, M\'exico, 62209. 

\noindent {\it E-mail address:} gabriela@uaem.mx 

\vskip .3cm
\noindent A. Verjovsky. {\tt Instituto de Matem\'aticas, Unidad Cuernavaca}. Universidad Nacional Au\-t\'o\-no\-ma de M\'exico.
Av. Universidad s/n, Col. Lomas de Chamilpa. Cuernavaca, Morelos, M\'exico, 62209.

\noindent {\it E-mail address:} alberto@matcuer.unam.mx

\end{document}